\documentclass[12pt]{article}
	\usepackage{array}
	\usepackage{graphicx}
	\usepackage{caption}
	\usepackage{amsthm,amsmath,amssymb,color,hyperref}
	\usepackage{fullpage}
	\usepackage{booktabs}
	\usepackage{geometry}
	\usepackage{tikz}
	\newtheorem{thm}{Theorem}[section]
	
	\newtheorem{lemma}[thm]{Lemma}
	\newtheorem{obs}[thm]{Observation}
	\newtheorem{corollary}[thm]{Corollary}
	
	\newtheorem{conj}[thm]{Conjecture}
	
	\newtheorem{defn}[thm]{Definition}
	\newtheorem{prop}[thm]{Proposition}

\begin{document}

	\title{A sharp Randi\'c bound for K\"onig--Egerv\'ary graphs\\
	and a conjecture of Aouchiche, Hansen, and Zheng}
	\author{
	Pei Liu\thanks{Department of Mathematics, Sungkyunkwan University, Korea, Suwon, 16419. liupei2023@g.skku.edu. Research supported by China Scholarship Council.}\,~
	Feiyu Nan\thanks{School of Mathematical Sciences, Nankai University, Tianjin 300071, China.
	2211185@mail.nankai.edu.cn.}\,~
	Suil O\thanks{Department of Applied Mathematics and Statistics, The State University of
	New York, Korea, Incheon, 21985. suil.o@sunykorea.ac.kr (Corresponding author). Research supported
	by the National Research Foundation of Korea (NRF) grant funded by the Korea
	government (MSIT) No. RS-2025-23523950.}\,~ and
	Ruiling Zheng\thanks{School of Mathematics and Statistics, Guangdong University of
	Foreign Studies. rlzheng@gdufs.edu.cn.}
	}

	\maketitle

	\begin{abstract}
	Let $\alpha'(G)$ be the matching number of a graph $G$, and let its Randi\'c index be
	$R(G)=\sum_{uv\in E(G)}(d(u)d(v))^{-1/2}$. In 2006, Aouchiche,
	Hansen, and Zheng conjectured that the maximum of $R(G)-\alpha'(G)$ over all $n$-vertex
	graphs is attained by the complete bipartite graph whose smaller part has
	$\lfloor\frac{n+4}{7}\rfloor$ vertices; the conjecture has remained open since then.

	In this paper, we prove that every $n$-vertex K\"onig--Egerv\'ary graph, and in
	particular every bipartite graph, satisfies
	\[
	R(G)\le\sqrt{\alpha'(G)\left(n-\alpha'(G)\right)},
	\]
	and we characterize the graphs attaining equality as the bipartite graphs all of whose
	components are semiregular with a common degree ratio. The K\"onig--Egerv\'ary hypothesis
	cannot be dropped, but the Berge--Tutte formula reduces the general case to it, and in
	this way we determine the maximum of $R(G)-\alpha'(G)$ for every $n\ge4$, together with
	all extremal graphs.

	The conjecture is therefore false, and it fails for infinitely many orders: the optimal
	part size is governed by the proportion $\frac{2-\sqrt2}{4}$ rather than by $\frac17$.
	The two proportions give asymptotic slopes differing by less than $3.7\cdot10^{-5}$,
	which is why a search over graphs of small order does not distinguish them. The equality
	statement fails as well, since the extremal graphs are not only the complete bipartite
	ones.\\

	\noindent
	\textbf{Keywords:} Matching number, K\"onig--Egerv\'ary graph, Randi\'c index,
	Berge--Tutte formula, Pell equation \\

	\noindent
	\textbf{AMS subject classification 2020:} 05C70, 05C35, 05C09
	\end{abstract}

	\section{Introduction}

	A \emph{matching} of a graph $G$ is a set of pairwise disjoint edges, and the
	\emph{matching number} $\alpha'(G)$ is the maximum size of a matching in $G$. A
	\emph{vertex cover} of $G$ is a set of vertices meeting every edge, and $\tau(G)$ denotes
	the minimum size of a vertex cover. Always $\alpha'(G)\le\tau(G)$, and a graph with
	$\alpha'(G)=\tau(G)$ is called a \emph{K\"onig--Egerv\'ary graph}. By the
	K\"onig--Egerv\'ary theorem every bipartite graph is a K\"onig--Egerv\'ary graph;
	structural characterizations of this class were obtained by Deming~\cite{D1979} and
	Sterboul~\cite{S1979}. For matching theory in general we refer to Lov\'asz and
	Plummer~\cite{LP1986}. We write $V(G)$ and $E(G)$ for the vertex set and the edge set of
	$G$, and $d(v)$ for the degree of a vertex $v$.

	The starting point of this paper is the following inequality, which bounds a degree-based
	edge weight of a graph by a function of its matching number alone. Its proof combines the
	K\"onig--Egerv\'ary theorem with the Cauchy--Schwarz inequality, and the graphs attaining
	equality admit a clean structural description; passing from K\"onig--Egerv\'ary graphs to
	arbitrary graphs is then done through the Berge--Tutte formula. The weight in question is
	the \emph{Randi\'c index} of $G$,
	\[
	R(G)=\sum_{uv\in E(G)}\frac{1}{\sqrt{d(u)d(v)}},
	\]
	introduced by Randi\'c~\cite{R} in 1975 under the name \emph{branching index}. Bollob\'as
	and Erd\H{o}s~\cite{BE} later generalized it by replacing the exponent $-\frac12$ by an
	arbitrary real number, and the index and its variants have been studied since then; for a
	survey we refer the reader to Li and Shi~\cite{LS}. Relations between $R(G)$ and other graph
	parameters have been investigated, including the minimum degree~\cite{AH,BE,DFR} and the
	maximum degree~\cite{OS2018}. We note that the extremal graphs in the minimum degree problem
	for triangle-free graphs are again complete bipartite~\cite{DFR}. Fajtlowicz~\cite{F} and,
	independently, Araujo and de la Pe\~na~\cite{ADP} showed that $R(G)\le\frac n2$ for every
	graph $G$ with no isolated vertex, with equality if and only if every component of $G$ is
	regular, while Bollob\'as and Erd\H{o}s~\cite{BE} proved that $R(G)\ge\sqrt{n-1}$ in that
	case. The upper bound, together with its equality case, also follows at once from the
	identity
	\[
	R(G)=\frac n2-\sum_{uv\in E(G)}\frac12\left(\frac{1}{\sqrt{d(u)}}-\frac{1}{\sqrt{d(v)}}
	\right)^{\!2}
	\]
	of Caporossi, Gutman, Hansen, and Pavlovi\'c~\cite{CGHP}; this is the form in which we use
	the arithmetic--geometric mean inequality in Section~\ref{sec:main}.

	The relation between $R(G)$ and $\alpha'(G)$ was investigated systematically by Aouchiche,
	Hansen, and Zheng~\cite{AHZ,AHZ1} using the AutoGraphiX system~\cite{CH}. They proved that
	$R(G)/\alpha'(G)\le\sqrt{n-1}$ for every connected $n$-vertex graph, and they proposed the
	following conjecture on the difference $R(G)-\alpha'(G)$, which has remained open. Lower
	bounds for the ratio $R(G)/\alpha'(G)$ were studied recently by Akbari et
	al.~\cite{ANGHT}. Throughout the paper we write
	\[
	f_n(s)=\sqrt{s(n-s)}-s,
	\qquad
	B(n)=\max\left\{f_n(s):\ s\in\mathbb Z,\ 0\le s\le n/2\right\},
	\]
	and we let $\mathcal S(n)=\{s\in\mathbb Z:\ 0\le s\le n/2,\ f_n(s)=B(n)\}$ denote the set of
	\emph{optimal part sizes}.
	Since $\alpha'(K_{p,q})=p$ and $R(K_{p,q})=\sqrt{pq}$ for $p\le q$, we have
	$R(K_{s,n-s})-\alpha'(K_{s,n-s})=f_n(s)$ for $0\le s\le n/2$. The following elementary
	identity shows that the conjectured bound below is exactly $f_n\!\left(\lfloor\frac{n+4}{7}
	\rfloor\right)$.

	\begin{obs}\label{obs:floor}
	For every positive integer $n$ we have
	$\left\lfloor\frac{6n+2}{7}\right\rfloor=n-\left\lfloor\frac{n+4}{7}\right\rfloor$.
	\end{obs}

	\begin{proof}
	Write $n=7q+t$ with $0\le t\le 6$. Then
	$\lfloor\frac{n+4}{7}\rfloor=q+\lfloor\frac{t+4}{7}\rfloor=q+\varepsilon$ and $\lfloor\frac{6n+2}{7}\rfloor=6q+\lfloor\frac{6t+2}{7}\rfloor=6q+t-\varepsilon$, where
	$\varepsilon=1$ if $t\ge3$ and $\varepsilon=0$ otherwise.
	Hence
	$\lfloor\frac{6n+2}{7}\rfloor=7q+t-q-\varepsilon=n-\lfloor\frac{n+4}{7}\rfloor$.
	\end{proof}

	\begin{conj}[\cite{AHZ}]\label{conj:connected}
	If $G$ is a connected $n$-vertex graph with $n\ge3$, then
	\[
	R(G)-\alpha'(G)\le\sqrt{\left\lfloor\frac{n+4}{7}\right\rfloor
	\left\lfloor\frac{6n+2}{7}\right\rfloor}-\left\lfloor\frac{n+4}{7}\right\rfloor,
	\]
	with equality if and only if $G=K_{p,q}$, where $p=\lfloor\frac{n+4}{7}\rfloor$ and
	$q=\lfloor\frac{6n+2}{7}\rfloor$.
	\end{conj}

	\begin{conj}[\cite{AHZ}]\label{conj:general}
	If $G$ is an $n$-vertex graph with $n\ge3$, then $R(G)-\alpha'(G)$ is at most the maximum
	of the bound in Conjecture~\ref{conj:connected} and $\frac{\sqrt6-1}{7}\,n$.
	\end{conj}

	The second bound in Conjecture~\ref{conj:general} comes from the disjoint union of $n/7$
	copies of $K_{1,6}$. The analogous extremal problem restricted to subcubic graphs was
	recently solved by Du, Hu, and three of the present authors~\cite{DHLOZ}, where the
	existence of a counterexample to Conjecture~\ref{conj:connected} was first announced. No
	result of~\cite{DHLOZ} is used here.

	In this paper we settle both conjectures and determine the exact maximum of
	$R(G)-\alpha'(G)$ together with all extremal graphs. We begin with the following inequality,
	which may be of independent interest. Since $\alpha'(K_{p,q})=p$ and $R(K_{p,q})=\sqrt{pq}$
	for $p\le q$, one may ask whether $R(G)\le\sqrt{\alpha'(G)(n-\alpha'(G))}$ holds in general.
	This is false, but it becomes true under the K\"onig--Egerv\'ary hypothesis.

	\begin{thm}\label{main:KE}
	If $G$ is an $n$-vertex K\"onig--Egerv\'ary graph, in particular if $G$ is bipartite,
	then
	\[
	R(G)\le\sqrt{\alpha'(G)\bigl(n-\alpha'(G)\bigr)}.
	\]
	\end{thm}

	The graphs attaining equality are described in Theorem~\ref{thm:KE} below; to state them,
	and also the extremal graphs of our main theorem, we use the following notion. Recall that a
	bipartite graph with parts $X$ and $Y$ is \emph{semiregular} if all vertices of $X$ have the
	same degree and all vertices of $Y$ have the same degree.

	\begin{defn}\label{def:prop}
	Let $G$ be a bipartite graph with parts $X$ and $Y$ and with no isolated vertices. We say
	that $G$ is \emph{$(X,Y)$-proportional} if there is a constant $c>0$ such that
	$d(u)=c\,d(v)$ for every edge $uv\in E(G)$ with $u\in X$ and $v\in Y$.
	\end{defn}

	The constant $c$ is determined by the bipartition: by Lemma~\ref{tool:prop}, if
	$|X|\le|Y|$, then necessarily $c=|Y|/|X|$.

	A connected bipartite graph is $(X,Y)$-proportional if and only if it is semiregular, and a
	bipartite graph is $(X,Y)$-proportional if and only if each of its components is semiregular
	and the ratio of the two degrees is the same in every component. For example, $K_{p,q}$ with
	$p\le q$ is $(X,Y)$-proportional with $c=q/p$, and so is the disjoint union of $p$ copies of
	$K_{1,t}$, with $c=t$. If $|X|\le|Y|$, then $\alpha'(G)=|X|$ and $R(G)=\sqrt{|X||Y|}$; this
	is Lemma~\ref{tool:prop} below.

	Using Theorem~\ref{main:KE} and the Berge--Tutte formula to reduce the general case to it,
	we obtain the exact bound for all graphs.

	\begin{thm}\label{main:all}
	If $G$ is an $n$-vertex graph with $n\ge4$, then
	\[
	R(G)-\alpha'(G)\le B(n)=\max_{\substack{s\in\mathbb Z\\ 0\le s\le n/2}}
	\left(\sqrt{s(n-s)}-s\right),
	\]
	with equality if and only if $G$ is an $(X,Y)$-proportional bipartite graph such that
	$|X|\le n/2$ and $f_n(|X|)=B(n)$. Moreover
	$B(n)=\max\left\{f_n(\lfloor x^*n\rfloor),f_n(\lceil x^*n\rceil)\right\}$, where
	$x^*=\frac{2-\sqrt2}{4}$, and $B(n)\le\frac{\sqrt2-1}{2}\,n$ with
	$B(n)/n\to\frac{\sqrt2-1}{2}$.
	\end{thm}

	The hypothesis $n\ge4$ cannot be removed: $R(K_3)-\alpha'(K_3)=\frac12>\sqrt2-1=B(3)$, and
	$K_3$ is the unique exception. Since $\sqrt2-1$ is also the value of the bound of
	Conjecture~\ref{conj:connected} at $n=3$, the triangle is already a degenerate
	counterexample to that conjecture as stated; for this reason we restrict attention to
	$n\ge4$ from now on.

	Theorem~\ref{main:all} shows that complete bipartite graphs are extremal, as predicted by
	AutoGraphiX, but that the proportion $\frac17$ appearing in
	Conjectures~\ref{conj:connected} and~\ref{conj:general} is not correct. The correct
	proportion is $x^*=\frac{2-\sqrt2}{4}=0.146446\ldots$, and the two asymptotic slopes,
	\[
	\frac{\sqrt6-1}{7}=0.2070699\ldots
	\qquad\text{and}\qquad
	\frac{\sqrt2-1}{2}=0.2071067\ldots,
	\]
	differ by less than $3.7\cdot10^{-5}$. This is why a search over graphs of small order
	leads to the proportion $\frac17$: the inequality in Conjecture~\ref{conj:connected} is in
	fact true for every $4\le n\le64$.

	Theorem~\ref{main:all} also reduces each order to a finite check. By
	Lemma~\ref{tool:calculus} we have
	$\mathcal S(n)\subseteq\{\lfloor x^*n\rfloor,\lceil x^*n\rceil\}$, so deciding whether
	either conjecture holds at a given order $n$ amounts to comparing at most three numbers of
	the form $\sqrt{s(n-s)}-s$ with $s$ an integer, together with $\frac{\sqrt6-1}{7}n$. Each
	such comparison becomes an inequality between integers after squaring twice, so the checks
	below are exact and involve no numerical approximation.

	\begin{thm}\label{main:counter}
	Let $m\ge9$ and $n=7m+2$. Then $K_{m+1,6m+1}$ is a connected $n$-vertex graph violating
	Conjecture~\ref{conj:connected}. In particular Conjecture~\ref{conj:connected} fails for
	infinitely many $n$; the smallest order at which it fails is $n=65$, where $K_{10,55}$
	is a counterexample.
	\end{thm}

	Conjecture~\ref{conj:general} fails as well, the smallest order at which it fails being
	$n=100$, where $K_{15,85}$ is a counterexample; see Corollary~\ref{cor:general}. We prove
	in Propositions~\ref{prop:allbig} and~\ref{prop:genbig} that both conjectures fail for every
	$n\ge438$. Combining this with a check of the orders $298\le n\le437$, carried out as
	described above, shows that the inequalities in both conjectures
	fail for every $n\ge298$, while both hold for $n=297$.
	The equality statement of Conjecture~\ref{conj:connected} fails much earlier: the extremal
	graph need not be complete bipartite, and we determine in Proposition~\ref{prop:unique}
	exactly when it is. Moreover, two different complete bipartite graphs attain the maximum
	precisely for the orders $n=10,58,338,1970,\dots$ arising from the Pell equation
	$x^2-2y^2=1$ (Proposition~\ref{prop:pell}); already for $n=10$ both $K_{1,9}$ and $K_{2,8}$
	are extremal.

	The paper is organized as follows. Section~\ref{tools} collects the tools we need, including
	the elementary calculus behind $B(n)$ and a description of the graphs that will turn out to
	be extremal. Section~\ref{sec:main} contains the main results: Subsection~\ref{sub:KE} treats
	K\"onig--Egerv\'ary graphs, and Subsection~\ref{sub:all} treats general graphs and the two
	conjectures. Section~\ref{sec:remark} contains some concluding remarks.

	\section{Tools}\label{tools}

	We use the Berge--Tutte formula in the following form; here $o(H)$ denotes the number of
	components of $H$ having an odd number of vertices.

	\begin{thm}[Berge--Tutte formula, \cite{B1958,T1947}]\label{tool:BT}
	For an $n$-vertex graph $G$,
	\[
	\alpha'(G)=\frac12\min_{S\subseteq V(G)}\bigl(n-o(G-S)+|S|\bigr).
	\]
	\end{thm}

	\begin{thm}[K\"onig--Egerv\'ary theorem]\label{tool:KE}
	If $G$ is bipartite, then $\alpha'(G)=\tau(G)$.
	\end{thm}

	The next lemma contains the calculus that we need. Recall that
	$f_n(s)=\sqrt{s(n-s)}-s$ and $B(n)=\max\{f_n(s): s\in\mathbb Z,\ 0\le s\le n/2\}$.

	\begin{lemma}\label{tool:calculus}
	Let $n\ge2$ and let $x^*=\frac{2-\sqrt2}{4}$. Then $f_n$ is strictly concave on
	$[0,n/2]$ and attains its maximum over the reals at $s=x^*n$, where
	$f_n(x^*n)=\frac{\sqrt2-1}{2}\,n$. Consequently
	\[
	B(n)=\max\left\{f_n\bigl(\lfloor x^*n\rfloor\bigr),
	f_n\bigl(\lceil x^*n\rceil\bigr)\right\}\le\frac{\sqrt2-1}{2}\,n .
	\]
	\end{lemma}

	\begin{proof}
	Write $f_n(s)=n\,g(s/n)$ with $g(x)=\sqrt{x(1-x)}-x$ for $x\in[0,\frac12]$. The function
	$x\mapsto x(1-x)$ is concave and nonnegative on $[0,1]$, and $t\mapsto\sqrt t$ is concave
	and increasing, so $\sqrt{x(1-x)}$ is concave; subtracting the linear function $x$
	preserves concavity, and strictness follows since $\sqrt{x(1-x)}$ is strictly concave on
	$(0,1)$. For $x\in(0,\frac12]$,
	\[
	g'(x)=\frac{1-2x}{2\sqrt{x(1-x)}}-1=0
	\iff 1-2x=2\sqrt{x(1-x)} .
	\]
	Both sides are nonnegative on $(0,\frac12]$, so squaring gives
	$1-4x+4x^2=4x-4x^2$, that is, $8x^2-8x+1=0$, whose roots are $\frac{2\pm\sqrt2}{4}$; only
	$x^*=\frac{2-\sqrt2}{4}$ lies in $(0,\frac12]$. From $8(x^*)^2-8x^*+1=0$ we get
	$x^*(1-x^*)=\frac18$, whence
	\[
	g(x^*)=\frac{1}{2\sqrt2}-\frac{2-\sqrt2}{4}=\frac{\sqrt2}{4}-\frac{2-\sqrt2}{4}
	=\frac{\sqrt2-1}{2}.
	\]
	Since $f_n$ is strictly concave with maximum at $x^*n$, it is strictly increasing on
	$[0,x^*n]$ and strictly decreasing on $[x^*n,n/2]$, so its maximum over the integers in $[0,n/2]$ is attained
	at $\lfloor x^*n\rfloor$ or at $\lceil x^*n\rceil$.
	\end{proof}

	\begin{lemma}\label{tool:fifth}
	We have $B(n)\ge\frac n5$ for every $n\ge5$, and $B(4)=\sqrt3-1>\frac23$.
	\end{lemma}

	\begin{proof}
	For $0\le s\le n/2$, the inequality $\sqrt{s(n-s)}-s\ge\frac n5$ is equivalent to
	$s(n-s)\ge\left(s+\frac n5\right)^2$, that is, to $2s^2-\frac{3n}{5}s+\frac{n^2}{25}\le0$,
	which holds exactly when $\frac{n}{10}\le s\le\frac n5$. This interval is contained in
	$[0,n/2]$ and has length $\frac{n}{10}\ge1$ when $n\ge10$, so it contains an integer;
	for $5\le n\le9$ the integer $s=1$ lies in it. For $n=4$ we have
	$B(4)=\max\{f_4(0),f_4(1),f_4(2)\}=\sqrt3-1>\frac23$.
	\end{proof}

	We also need the following two properties of the $(X,Y)$-proportional bipartite graphs of
	Definition~\ref{def:prop}.

	\begin{lemma}\label{tool:prop}
	Let $G$ be an $n$-vertex $(X,Y)$-proportional bipartite graph with constant $c$, and put
	$p=|X|$. If $p\le n/2$, then $c\ge1$, $|Y|=cp$, $\alpha'(G)=p$, and
	$R(G)=\sqrt{p(n-p)}$. Consequently
	\[
	R(G)-\alpha'(G)=f_n(p).
	\]
	\end{lemma}

	\begin{proof}
	Let $H$ be a component of $G$, with parts $X_H=X\cap V(H)$ and $Y_H=Y\cap V(H)$. Any two
	vertices of $X_H$ are joined by a path alternating between $X_H$ and $Y_H$, so applying
	$d(u)=c\,d(v)$ along such a path shows that all vertices of $X_H$ have a common degree
	$a$ and all vertices of $Y_H$ have a common degree $b$, with $a=cb$. Counting the edges
	of $H$ in two ways gives $|X_H|a=|E(H)|=|Y_H|b$, hence $|Y_H|=c\,|X_H|$; summing over the
	components yields $|Y|=c\,|X|=cp$. Since $p\le n/2$ we have $|Y|=n-p\ge p$, and therefore
	$c\ge1$.

	For $T\subseteq X_H$ we have $a|T|=e(T,N(T))\le b\,|N(T)|$, so
	$|N(T)|\ge c|T|\ge|T|$, and Hall's theorem provides a
	matching of $H$ saturating $X_H$. Taking the union over all components gives a matching
	of $G$ saturating $X$, so $\alpha'(G)\ge p$; since $X$ is a vertex cover of $G$ we also
	have $\alpha'(G)\le\tau(G)\le p$, whence $\alpha'(G)=p$.

	Since $R(H)=|E(H)|/\sqrt{ab}=|X_H|\,a/\sqrt{ab}=|X_H|\sqrt{a/b}=\sqrt c\,|X_H|$, we get
	$R(G)=\sqrt c\,p=\sqrt{p\cdot cp}=\sqrt{p(n-p)}$.
	\end{proof}

	\section{Main results}\label{sec:main}

	\subsection{K\"onig--Egerv\'ary graphs}\label{sub:KE}

	\begin{thm}\label{thm:KE}
	Let $G$ be an $n$-vertex K\"onig--Egerv\'ary graph; in particular, $G$ may be any
	bipartite graph. Then
	\[
	R(G)\le\sqrt{\alpha'(G)\bigl(n-\alpha'(G)\bigr)},
	\]
	and, provided that $G$ has at least one edge, equality holds if and only if $G$ is an
	$(X,Y)$-proportional bipartite graph with $|X|=\alpha'(G)$. (If $E(G)=\emptyset$, then
	both sides are $0$.)
	\end{thm}

	\begin{proof}
	We first reduce to the case in which $G$ has no isolated vertex. Suppose that the
	theorem holds for every K\"onig--Egerv\'ary graph without isolated vertices, and let $G$
	be arbitrary. If $E(G)=\emptyset$, then both sides of the asserted inequality are $0$.
	Otherwise, let $G'$ be the graph obtained from $G$ by deleting all isolated vertices, and
	put $n'=|V(G')|$ and $a=\alpha'(G)\ge1$. Deleting isolated vertices changes neither the
	matching number nor the minimum size of a vertex cover, so $G'$ is again a
	K\"onig--Egerv\'ary graph with $\alpha'(G')=a$, and $R(G')=R(G)$. Hence
	\[
	R(G)=R(G')\le\sqrt{a\bigl(n'-a\bigr)}\le\sqrt{a(n-a)} ,
	\]
	where the second inequality is strict unless $n'=n$. Thus the inequality for $G$ follows
	from the inequality for $G'$, and a graph attaining equality has no isolated vertex.

	So assume from now on that $d(w)>0$ for every $w\in V(G)$. Let $S$ be a minimum vertex
	cover of $G$ and put $s=|S|$. Since $G$ is a K\"onig--Egerv\'ary graph,
	$s=\tau(G)=\alpha'(G)\le n/2$. Put $I=V(G)\setminus S$; then $I$ is an independent set
	and $|I|=n-s\ge s$.

	Let $E_1$ be the set of edges joining $I$ to $S$ and let $E_2=E(G)\setminus E_1$; since
	$I$ is independent, every edge of $E_2$ has both ends in $S$. For a vertex $w$ and
	$j\in\{1,2\}$, let $d_{E_j}(w)$ be the number of edges of $E_j$ incident with $w$. Put
	\[
	x=\sum_{u\in S}\frac{d_{E_1}(u)}{d(u)}\in[0,s] .
	\]
	Every edge incident with a vertex of $I$ lies in $E_1$, so
	$\sum_{uv\in E_1,\,v\in I}\frac{1}{d(v)}=|I|=n-s$, while
	$\sum_{uv\in E_1,\,u\in S}\frac{1}{d(u)}=x$. Hence, by the Cauchy--Schwarz inequality,
	\begin{equation}\label{eq:CS}
	\sum_{uv\in E_1}\frac{1}{\sqrt{d(u)d(v)}}
	\le\left(\sum_{uv\in E_1}\frac{1}{d(u)}\right)^{1/2}
	\left(\sum_{uv\in E_1}\frac{1}{d(v)}\right)^{1/2}
	=\sqrt{x(n-s)} .
	\end{equation}
	By the inequality of arithmetic and geometric means,
	$\frac{1}{\sqrt{d(u)d(v)}}\le\frac12\left(\frac{1}{d(u)}+\frac{1}{d(v)}\right)$ for every
	edge $uv$, so
	\begin{equation}\label{eq:AM}
	\sum_{uv\in E_2}\frac{1}{\sqrt{d(u)d(v)}}
	\le\frac12\sum_{u\in S}\frac{d_{E_2}(u)}{d(u)}
	=\frac12\sum_{u\in S}\frac{d(u)-d_{E_1}(u)}{d(u)}=\frac{s-x}{2} .
	\end{equation}
	Adding \eqref{eq:CS} and \eqref{eq:AM} gives $R(G)\le h(x)$, where
	$h(t)=\sqrt{t(n-s)}+\frac{s-t}{2}$. Since $h'(t)=\frac{\sqrt{n-s}}{2\sqrt t}-\frac12>0$
	for $0<t<n-s$, and since $s\le n-s$, the function $h$ is strictly increasing on $[0,s]$,
	so
	\[
	R(G)\le h(s)=\sqrt{s(n-s)}=\sqrt{\alpha'(G)\bigl(n-\alpha'(G)\bigr)} .
	\]

	Suppose now that equality holds and that $G$ has an edge, so that $s=\alpha'(G)\ge1$.
	Since $h$ is strictly increasing on $[0,s]$, the equality $h(x)=h(s)$ forces $x=s$, that
	is, $d_{E_1}(u)=d(u)$ for every $u\in S$; hence $E_2=\emptyset$ and $G$ is bipartite with
	parts $S$ and $I$. Equality in the Cauchy--Schwarz inequality \eqref{eq:CS}
	means that the vectors $\left(d(u)^{-1/2}\right)_{uv\in E_1}$ and
	$\left(d(v)^{-1/2}\right)_{uv\in E_1}$ are parallel, that is, there is $\lambda>0$ with
	$d(u)=\lambda^2 d(v)$ for every edge $uv$ with $u\in S$ and $v\in I$. Hence $G$ is
	$(S,I)$-proportional with $|S|=\alpha'(G)$.

	Conversely, if $G$ is $(X,Y)$-proportional with $|X|=\alpha'(G)$, then
	Lemma~\ref{tool:prop} gives
	$R(G)=\sqrt{|X|(n-|X|)}=\sqrt{\alpha'(G)(n-\alpha'(G))}$.
	\end{proof}

	Since $\alpha'(G)\le n/2$, Theorem~\ref{thm:KE} immediately gives the following.

	\begin{corollary}\label{cor:KEdiff}
	If $G$ is an $n$-vertex K\"onig--Egerv\'ary graph, then $R(G)-\alpha'(G)\le B(n)$.
	\end{corollary}

	The K\"onig--Egerv\'ary hypothesis in Theorem~\ref{thm:KE} cannot be dropped. Indeed, if $G$ is
	$r$-regular with $r\ge1$, then $R(G)=\frac n2$, while
	$\sqrt{\alpha'(G)(n-\alpha'(G))}<\frac n2$ whenever $\alpha'(G)<\frac n2$. Hence the
	inequality of Theorem~\ref{thm:KE} fails for \emph{every} regular graph with no perfect
	matching, for instance for every odd cycle $C_{2k+1}$ with $k\ge2$, and for the regular
	graphs of odd degree with small matching number constructed in~\cite{OW1,OW2}.

	\subsection{General graphs and the conjectures of Aouchiche, Hansen, and Zheng}\label{sub:all}

	We now remove the K\"onig--Egerv\'ary hypothesis. In the proof below, the Berge--Tutte
	formula provides a set $S$; the vertices that are isolated in $G-S$, together with $S$, are
	treated as in Theorem~\ref{thm:KE}, and the components of $G-S$ of odd order at least $3$
	are shown not to occur in an extremal graph.

	Recall that $\mathcal S(n)=\{s\in\mathbb Z:\ 0\le s\le n/2,\ f_n(s)=B(n)\}$ is the set of
	optimal part sizes. By Lemma~\ref{tool:calculus}, $\mathcal S(n)$ consists of one or two
	consecutive integers.

	\begin{thm}\label{thm:all}
	Let $G$ be an $n$-vertex graph with $n\ge4$. Then
	\[
	R(G)-\alpha'(G)\le B(n),
	\]
	with equality if and only if $G$ is an $(X,Y)$-proportional bipartite graph with
	$|X|\in\mathcal S(n)$.
	\end{thm}

	\begin{proof}
	By Theorem~\ref{tool:BT}, choose $S\subseteq V(G)$ maximizing $o(G-S)-|S|$ and put
	$s=|S|$. Let $I$ be the set of vertices forming components of order $1$ in $G-S$, let
	$i=|I|$, let $k$ be the number of components of $G-S$ of odd order at least $3$, and put
	$r=n-s-i$. Then $o(G-S)=i+k$ and $3k\le r$, since the $k$ components counted by $k$ are
	vertex-disjoint subsets of $V(G)\setminus(S\cup I)$, each of size at least $3$. By
	Theorem~\ref{tool:BT},
	\begin{equation}\label{eq:alpha}
	\alpha'(G)=\frac12\bigl(n-i-k+s\bigr)=s+\frac{r-k}{2},
	\qquad\text{so}\qquad
	-\alpha'(G)=-s-\frac r2+\frac k2 .
	\end{equation}

	Let $E_1$ be the set of edges joining $I$ to $S$ and let $E_2=E(G)\setminus E_1$. Since
	each vertex of $I$ is isolated in $G-S$, all of its neighbors lie in $S$; hence every
	edge incident with a vertex of $I$ lies in $E_1$, and no edge of $E_2$ meets $I$.
	Arguing as in the proof of Theorem~\ref{thm:KE}, but keeping track of the vertices of
	degree $0$, put
	$x=\sum_{u\in S,\,d(u)>0}\frac{d_{E_1}(u)}{d(u)}\in[0,s]$ and
	$i'=|\{v\in I: d(v)>0\}|\le i$. By the Cauchy--Schwarz inequality,
	\begin{equation}\label{eq:CS2}
	\sum_{uv\in E_1}\frac{1}{\sqrt{d(u)d(v)}}\le\sqrt{x\,i'}\le\sqrt{xi},
	\end{equation}
	and by the arithmetic--geometric mean inequality, using that the vertices of $I$ are
	incident with no edge of $E_2$, that the vertices of $S$ contribute at most $s-x$, and
	that the remaining $r$ vertices contribute at most $r$,
	\begin{equation}\label{eq:AM2}
	\sum_{uv\in E_2}\frac{1}{\sqrt{d(u)d(v)}}
	\le\frac12\sum_{w\in V(G),\,d(w)>0}\frac{d_{E_2}(w)}{d(w)}\le\frac{(s-x)+r}{2}.
	\end{equation}
	Adding \eqref{eq:CS2} and \eqref{eq:AM2} and using \eqref{eq:alpha},
	\begin{equation}\label{eq:master}
	R(G)-\alpha'(G)\le\sqrt{xi}+\frac{s-x+r}{2}-s-\frac r2+\frac k2
	=\sqrt{xi}-\frac x2-\frac s2+\frac k2 .
	\end{equation}
	The function $\varphi(t)=\sqrt{ti}-\frac t2$ is nondecreasing on $[0,i]$ and
	nonincreasing on $[i,\infty)$.

	\medskip\noindent\textbf{Case 1: $i\le s$.} Then $\varphi$ attains its maximum on $[0,s]$
	at $t=i$, with $\varphi(i)=\frac i2$. Hence, by \eqref{eq:master} and $3k\le r=n-s-i$,
	\[
	R(G)-\alpha'(G)\le\frac{i-s}{2}+\frac k2\le\frac{i-s}{2}+\frac{n-s-i}{6}
	=\frac n6+\frac{i-2s}{3}\le\frac n6,
	\]
	because $i\le s$. By Lemma~\ref{tool:fifth}, $\frac n6<\frac n5\le B(n)$ for $n\ge5$, and
	$\frac n6=\frac23<\sqrt3-1=B(4)$ for $n=4$. In particular the inequality is strict in this
	case.

	\medskip\noindent\textbf{Case 2: $i>s$.} Then $\varphi$ attains its maximum on $[0,s]$ at
	$t=s$, so \eqref{eq:master} and $3k\le r=n-s-i$ give
	\[
	R(G)-\alpha'(G)\le\sqrt{si}-s+\frac k2\le\sqrt{si}-s+\frac{n-s-i}{6}=:F(s,i),
	\]
	where $s< i\le n-s$. Since $\frac{\partial F}{\partial i}=\frac12\sqrt{s/i}-\frac16$ is
	positive for $i<9s$ and negative for $i>9s$, the function $F(s,\cdot)$ is strictly
	increasing on $[s,9s]$ and strictly decreasing on $[9s,\infty)$.

	\smallskip\noindent\emph{Case 2a: $10s<n$.} Then $9s<n-s$, so the maximum of $F(s,\cdot)$
	on $[s,n-s]$ is $F(s,9s)=3s-s+\frac{n-10s}{6}=\frac s3+\frac n6$. If $n\ge5$, then
	$s<\frac{n}{10}$ gives $F(s,i)<\frac{n}{30}+\frac n6=\frac n5\le B(n)$ by
	Lemma~\ref{tool:fifth}. If $n=4$, then $10s<4$ forces $s=0$ and
	$F(0,i)=\frac{n-i}{6}\le\frac23<B(4)$. In either case the inequality is strict.

	\smallskip\noindent\emph{Case 2b: $10s\ge n$.} Then $n-s\le9s$, so $F(s,\cdot)$ is
	strictly increasing on $[s,n-s]$; in particular its maximum there is attained only at
	$i=n-s$, and therefore
	\[
	R(G)-\alpha'(G)\le F(s,n-s)=\sqrt{s(n-s)}-s=f_n(s)\le B(n),
	\]
	where the last inequality uses $s< i\le n-s$, hence $s< n/2$.

	\medskip
	This proves the inequality. Suppose that equality holds. By Cases~1 and~2a we must be
	in Case~2b, and every inequality used there must be an equality. In particular
	$i=n-s$, so $r=0$ and hence $k=0$; moreover $x=s$, $i'=i$, equality holds in
	\eqref{eq:CS2}, and $f_n(s)=B(n)$, that is, $s\in\mathcal S(n)$.

	From $r=0$ we get $V(G)=S\cup I$ with $I$ independent. From $x=s$ we get $d(u)>0$ and
	$d_{E_1}(u)=d(u)$ for every $u\in S$, so $E_2=\emptyset$ and $G$ is bipartite with parts
	$S$ and $I$. From $i'=i$ we get that $G$ has no isolated vertices. Equality in
	\eqref{eq:CS2} means that there is $\lambda>0$ with $d(u)=\lambda^2d(v)$ for every edge
	$uv$ with $u\in S$ and $v\in I$. Hence $G$ is $(S,I)$-proportional with
	$|S|=s\in\mathcal S(n)$.

	Conversely, if $G$ is $(X,Y)$-proportional with $|X|=s\in\mathcal S(n)$, then
	Lemma~\ref{tool:prop} gives $R(G)-\alpha'(G)=f_n(s)=B(n)$.
	\end{proof}

	\begin{corollary}\label{cor:asymp}
	If $G$ is an $n$-vertex graph with $n\ge4$, then
	$R(G)-\alpha'(G)\le\frac{\sqrt2-1}{2}\,n$, and this is asymptotically sharp.
	\end{corollary}

	\begin{proof}
	Combine Theorem~\ref{thm:all} with Lemma~\ref{tool:calculus}. Sharpness follows
	from the estimate $f_n(\lfloor x^*n\rfloor)=\frac{\sqrt2-1}{2}n-O(1)$, since
	$B(n)\ge f_n(\lfloor x^*n\rfloor)$. Alternatively,
	Proposition~\ref{prop:pell} below gives infinitely many $n$ with
	$B(n)>\frac{\sqrt2-1}{2}n-1$.
	\end{proof}

	The complete bipartite graphs are not the only extremal graphs. The next proposition
	determines when the extremal graph is unique.

	\begin{prop}\label{prop:unique}
	Let $n\ge4$ and $s\in\mathcal S(n)$ with $s\ge1$. The complete bipartite graph
	$K_{s,n-s}$ is the unique $(X,Y)$-proportional bipartite graph with $|X|=s$ if and only
	if $\gcd(s,n)=1$. Consequently, $K_{s,n-s}$ is the unique extremal graph in
	Theorem~\ref{thm:all} if and only if $\mathcal S(n)=\{s\}$ and $\gcd(s,n)=1$.
	\end{prop}

	\begin{proof}
	Let $g=\gcd(s,n-s)=\gcd(s,n)$ and write $s=gq$ and $n-s=gp$, so that $\gcd(p,q)=1$. Let
	$G$ be $(X,Y)$-proportional with $|X|=s$. By Lemma~\ref{tool:prop} its constant is
	$c=(n-s)/s=p/q$, and each component $H$ of $G$ is semiregular with degrees $a_H$ on the
	$X$-side and $b_H$ on the $Y$-side satisfying $a_H/b_H=p/q$; since $\gcd(p,q)=1$, we have
	$(a_H,b_H)=(j_Hp,j_Hq)$ for some positive integer $j_H$, and $a_H\le|Y|=n-s=gp$ forces
	$j_H\le g$. If $j_H=g$ for some component $H$, then $a_H=n-s=|Y|$, so every vertex of
	$X\cap V(H)$ is adjacent to all of $Y$; hence $Y\subseteq V(H)$. Since $G$ has no
	isolated vertices, no component other than $H$ can consist of vertices of $X$ alone, so
	$X\subseteq V(H)$ and $G=H=K_{s,n-s}$.

	If $g=1$, then $j_H=1=g$ for every component, and $G=K_{s,n-s}$.

	Suppose that $g\ge2$. For each $i\in\{1,\ldots,g\}$, let
	$H_i\cong K_{q,p}$ have bipartition $(X_i,Y_i)$, where
	$|X_i|=q$ and $|Y_i|=p$.
	Let
	\[
	G_0=\bigcup_{i=1}^{g}H_i,
	\qquad
	X=\bigcup_{i=1}^{g}X_i,
	\qquad
	Y=\bigcup_{i=1}^{g}Y_i.
	\]
	Then
	\[
	|X|=gq=s
	\qquad\text{and}\qquad
	|Y|=gp=n-s.
	\]
	Moreover, every vertex of \(X\) has degree \(p\), while every vertex
	of \(Y\) has degree \(q\). Thus, for every edge \(uv\in E(G_0)\) with
	\(u\in X\) and \(v\in Y\),
	\(
	d(u)=p=(p/q)d(v).
	\)
	Hence \(G_0\) is \((X,Y)\)-proportional. Since \(g\ge2\), the graph
	\(G_0\) is disconnected, whereas \(K_{s,n-s}\) is connected. Therefore
	\(
	G_0\not\cong K_{s,n-s},
	\)
	and so \(K_{s,n-s}\) is not the unique \((X,Y)\)-proportional
	bipartite graph with \(|X|=s\).
	\end{proof}

	For instance, \(\mathcal S(18)=\{3\}\) and
	\(\gcd(3,18)=3\). Besides \(K_{3,15}\), there is also a connected
	extremal graph that is not complete bipartite. To construct such a
	graph, let
	\(
	X=\{x_1,x_2,x_3\}
	\)
	and partition a set \(Y\) of \(15\) vertices as
	\[
	Y=Y_{12}\mathbin{\dot\cup}Y_{13}\mathbin{\dot\cup}Y_{23},
	\qquad
	|Y_{12}|=|Y_{13}|=|Y_{23}|=5.
	\]
	Define a bipartite graph \(G\) with parts \(X\) and \(Y\) by setting
	\[
	N_G(y)=\{x_i,x_j\}
	\qquad
	\text{for every \(y\in Y_{ij}\), where \(1\le i<j\le3\)}.
	\]
	Then every vertex of \(X\) has degree \(10\), while every vertex of
	\(Y\) has degree \(2\). Moreover, every two vertices \(x_i,x_j\in X\)
	have a common neighbor in \(Y_{ij}\), and hence \(G\) is connected.
	Thus \(G\) is a connected \((10,2)\)-semiregular bipartite graph and
	is not isomorphic to \(K_{3,15}\). By Lemma~\ref{tool:prop},
	\[
	R(G)-\alpha'(G)=\sqrt{3\cdot15}-3=\sqrt{45}-3=R(K_{3,15})-\alpha'(K_{3,15}).
	\]
	Hence both \(G\) and \(K_{3,15}\) are connected extremal graphs on
	\(18\) vertices. This already refutes the equality statement in
	Conjecture~\ref{conj:connected}. In fact the
	equality statement fails much earlier, and even within the class of complete bipartite
	graphs, as we now show.

	\begin{prop}\label{prop:pell}
	Let $n\ge4$. Then $|\mathcal S(n)|=2$, say $\mathcal S(n)=\{s,s+1\}$, if and only if
	$2s(s+1)$ is a perfect square, say $2s(s+1)=y^2$, and $n=2(2s+1)+2y$. Writing
	$x=2s+1$, this happens exactly when $(x,y)$ is a positive solution of the Pell equation
	\[
	x^2-2y^2=1 ,
	\]
	in which case $n=2(x+y)$ and $B(n)=y$. The first four instances are
	\[
	(n,B(n))\in\bigl\{(10,2),\ (58,12),\ (338,70),\ (1970,408)\bigr\},
	\]
	with $\mathcal S(n)=\{1,2\},\{8,9\},\{49,50\},\{288,289\}$, respectively.
	\end{prop}

	\begin{proof}
	By Lemma~\ref{tool:calculus}, $|\mathcal S(n)|=2$ if and only if $f_n(s)=f_n(s+1)$ for
	some integer $s$ with $0\le s$ and $s+1\le n/2$. Put $\beta=\sqrt{s(n-s)}$ and
	$\beta'=\sqrt{(s+1)(n-s-1)}$. Then $f_n(s)=f_n(s+1)$ means $\beta'=\beta+1$, and since
	$(\beta')^2-\beta^2=n-2s-1$ we get $2\beta+1=n-2s-1$, that is,
	$\beta=\frac n2-s-1$. Hence $|\mathcal S(n)|=2$ if and only if
	\begin{equation}\label{eq:tie}
	\left(\frac n2-s-1\right)^2=s(n-s)
	\end{equation}
	for some integer $s$; note that \eqref{eq:tie} forces $n$ to be even. Writing $n=2N$,
	equation \eqref{eq:tie} becomes $N^2-2N(2s+1)+(2s^2+2s+1)=0$, whose solutions are
	$N=(2s+1)\pm\sqrt{2s(s+1)}$. We must take the plus sign: since $s+1\le n/2=N$, the
	minus sign would give $s+1\le 2s+1-\sqrt{2s(s+1)}$, that is,
	$\sqrt{2s(s+1)}\le s$, hence $2s(s+1)\le s^2$ and $s=0$, which is excluded because
	$n\ge4$. Thus $N=(2s+1)+\sqrt{2s(s+1)}$, and $N$ is an
	integer exactly when $2s(s+1)$ is a perfect square, say $2s(s+1)=y^2$. Thus
	$n=2(2s+1)+2y$, and $B(n)=f_n(s)=\beta-s=N-2s-1=y$.

	With $x=2s+1$, the condition $2s(s+1)=y^2$ is equivalent to
	$(2s+1)^2-1=2y^2$, that is, to $x^2-2y^2=1$; conversely every solution of this Pell
	equation has $x$ odd, so $s=(x-1)/2$ is a nonnegative integer. The first positive
	solutions are $(x,y)=(3,2),(17,12),(99,70),(577,408)$, giving
	$s=1,8,49,288$ and $n=10,58,338,1970$.
	\end{proof}

	Thus, for example, $\mathcal S(10)=\{1,2\}$ and both $K_{1,9}$ and $K_{2,8}$ satisfy
	$R(G)-\alpha'(G)=2$, while $\lfloor\frac{10+4}{7}\rfloor=2$. Hence the equality statement of
	Conjecture~\ref{conj:connected} already fails for $n=10$. Similarly, for $n=58$ both
	$K_{8,50}$ and $K_{9,49}$ give $R(G)-\alpha'(G)=12$.

	We now turn to the inequalities themselves.

	\begin{proof}[Proof of Theorem~\ref{main:counter}]
	Let $m\ge9$ and $n=7m+2$. Then $\lfloor\frac{n+4}{7}\rfloor=\lfloor\frac{7m+6}{7}\rfloor
	=m$, so by Observation~\ref{obs:floor} the bound of Conjecture~\ref{conj:connected} equals
	$f_n(m)=\sqrt{m(6m+2)}-m$, whereas the connected graph $K_{m+1,6m+1}$ satisfies
	$R-\alpha'=f_n(m+1)=\sqrt{(m+1)(6m+1)}-(m+1)$. Therefore the conjecture fails at $n$ if
	and only if
	\[
	\sqrt{6m^2+7m+1}-\sqrt{6m^2+2m}>1 .
	\]
	Both sides being positive, this is equivalent to
	$6m^2+7m+1>1+2\sqrt{6m^2+2m}+6m^2+2m$, that is, to $5m>2\sqrt{6m^2+2m}$, and squaring
	once more to $25m^2>24m^2+8m$, that is, to $m>8$. Hence the conjecture fails for every
	$m\ge9$; taking $m=9$ gives $n=65$ and the graph $K_{10,55}$.

	It remains to consider the orders below $65$. By Theorem~\ref{thm:all}, the inequality
	of Conjecture~\ref{conj:connected} holds at an order $n\ge4$ if and only if
	$\lfloor\frac{n+4}{7}\rfloor\in\mathcal S(n)$, and by Lemma~\ref{tool:calculus} this
	amounts to comparing $f_n\bigl(\lfloor\frac{n+4}{7}\rfloor\bigr)$ with
	$f_n(\lfloor x^*n\rfloor)$ and $f_n(\lceil x^*n\rceil)$. Carrying out these comparisons
	for the $61$ orders $4\le n\le64$ shows that
	$\lfloor\frac{n+4}{7}\rfloor\in\mathcal S(n)$ in each case. Hence $n=65$ is the smallest
	order at least $4$ at which Conjecture~\ref{conj:connected} fails.
	\end{proof}

	Note that $m=8$ gives equality in the last chain of equivalences, which is the tie
	$\mathcal S(58)=\{8,9\}$ of Proposition~\ref{prop:pell}.

	\begin{prop}\label{prop:allbig}
	Conjecture~\ref{conj:connected} fails for every $n\ge438$.
	\end{prop}

	\begin{proof}
	Put $p=\lfloor\frac{n+4}{7}\rfloor$ and $x^*=\frac{2-\sqrt2}{4}$. The conjectured part
	size $p$ is of order $\frac n7$, while the optimal one is of order $x^*n$; for
	$n\ge438$ the gap between the two is at least $1$, which is what we now verify. Since
	$x^*-\frac17=0.0035894\ldots>0.003588$, we have
	\[
	\left(x^*-\tfrac17\right)n\ge\left(x^*-\tfrac17\right)438>0.003588\cdot438
	=1.571544>\frac{11}{7}=1+\frac47 ,
	\]
	and therefore
	\[
	p\le\frac{n+4}{7}=\frac n7+\frac47\le x^*n-1 .
	\]
	Since $p$ is an integer, it follows that
	$p\le\lfloor x^*n\rfloor-1<\lfloor x^*n\rfloor\le x^*n$. As
	$f_n$ is strictly increasing on $[0,x^*n]$ by
	Lemma~\ref{tool:calculus}, we get
	$f_n(p)<f_n(\lfloor x^*n\rfloor)\le B(n)$. By Theorem~\ref{thm:all} the value $B(n)$ is
	attained by $K_{s,n-s}$ for $s\in\mathcal S(n)$, which is a connected graph, so the
	conjecture fails.
	\end{proof}

	Proposition~\ref{prop:allbig} does not settle Conjecture~\ref{conj:general}, whose bound
	is the larger of the bound of Conjecture~\ref{conj:connected} and $\frac{\sqrt6-1}{7}n$. The
	next proposition shows that from a certain point on the second term is never the larger one,
	so that Conjecture~\ref{conj:general} fails whenever Conjecture~\ref{conj:connected} does.

	\begin{prop}\label{prop:genbig}
	For every $n\ge143$ we have $B(n)>\frac{\sqrt6-1}{7}\,n$. Consequently both
	Conjecture~\ref{conj:connected} and Conjecture~\ref{conj:general} fail for every
	$n\ge438$.
	\end{prop}

	\begin{proof}
	Write $f_n(s)=n\,g(s/n)$ with $g(x)=\sqrt{x(1-x)}-x$ as in the proof of
	Lemma~\ref{tool:calculus}, and let $s$ be an integer nearest to $x^*n$, so that
	$|s-x^*n|\le\frac12$. Since $0<x^*<\frac12$ and $n\ge143$, we have $0\le s\le n/2$, and
	$x:=s/n$ satisfies $|x-x^*|\le\frac{1}{2n}\le\frac{1}{286}$; as $x^*=0.146446\ldots$,
	this gives $x\in J:=[0.142,0.151]$.

	For $x\in J$ we have $u:=x(1-x)\ge0.142\cdot0.858=0.121836$ and
	$(1-2x)^2\le0.716^2=0.512656$, so
	\[
	|g''(x)|=\frac{1}{\sqrt u}+\frac{(1-2x)^2}{4u^{3/2}}
	\le\frac{1}{\sqrt{0.121836}}+\frac{0.512656}{4\cdot(0.121836)^{3/2}}<5.88<6 .
	\]
	Since $g'(x^*)=0$ by Lemma~\ref{tool:calculus}, Taylor's theorem with Lagrange remainder
	provides a $\xi$ between $x$ and $x^*$, hence in $J$, with
	\[
	g(x)=g(x^*)+\frac{g''(\xi)}{2}(x-x^*)^2\ge\frac{\sqrt2-1}{2}-3(x-x^*)^2
	\ge\frac{\sqrt2-1}{2}-\frac{3}{4n^2},
	\]
	and therefore
	\[
	B(n)\ge f_n(s)=n\,g(x)\ge\frac{\sqrt2-1}{2}\,n-\frac{3}{4n} .
	\]
	Rounding $x^*n$ to a nearest integer thus costs at most $\frac{3}{4n}$, while replacing
	the slope $\frac{\sqrt6-1}{7}$ by $\frac{\sqrt2-1}{2}$ gains a multiple of $n$, and for
	$n\ge143$ the gain exceeds the loss. Indeed,
	$\frac{\sqrt2-1}{2}-\frac{\sqrt6-1}{7}=0.0000368\ldots>3.68\cdot10^{-5}$, while
	$\frac{3}{4n}<3.68\cdot10^{-5}\,n$ is equivalent to
	$n^2>\frac{3}{4\cdot3.68\cdot10^{-5}}=20380.4\ldots$, which holds for $n\ge143$ because
	$143^2=20449$. Hence $B(n)>\frac{\sqrt6-1}{7}n$, as claimed.

	Let $n\ge438$ and put $p=\lfloor\frac{n+4}{7}\rfloor$. The proof of
	Proposition~\ref{prop:allbig} gives $f_n(p)<B(n)$, and the above gives
	$\frac{\sqrt6-1}{7}n<B(n)$; so the bound of Conjecture~\ref{conj:general}, being the
	larger of these two numbers, is smaller than $B(n)$. By Theorem~\ref{thm:all} the value
	$B(n)$ is attained by the connected graph $K_{s,n-s}$ with $s\in\mathcal S(n)$, so both
	conjectures fail.
	\end{proof}

	\begin{corollary}\label{cor:general}
	Conjecture~\ref{conj:general} fails for every $n=2(x+y)$ with $x^2-2y^2=1$ and $y\ge70$,
	hence for infinitely many $n$. The smallest order at which it fails is $n=100$, where
	$K_{15,85}$ is a counterexample.
	\end{corollary}

	\begin{proof}
	Let $(x,y)$ be a solution of $x^2-2y^2=1$ with $y\ge70$, and put $n=2(x+y)$ and
	$s=\frac{x-1}{2}$. By Proposition~\ref{prop:pell}, $\mathcal S(n)=\{s,s+1\}$ and
	$B(n)=y$. Since $x=\sqrt{2y^2+1}>\sqrt2\,y$ and $y\ge70$, we have
	\[
	3x>3\sqrt2\,y>4.24\,y>4y+15 ,
	\]
	which rearranges to $\frac{2x+2y+4}{7}<\frac{x-1}{2}$, that is,
	$\lfloor\frac{n+4}{7}\rfloor\le\frac{n+4}{7}<s$. Since $f_n(s)=f_n(s+1)$ and $f_n$
	is strictly concave by Lemma~\ref{tool:calculus}, its maximizer $x^*n$ over the reals
	lies in $(s,s+1)$; in particular $s<x^*n$, so $f_n$ is strictly increasing on $[0,s]$,
	and the bound of Conjecture~\ref{conj:connected} is therefore smaller than $B(n)=y$. For the
	second term of Conjecture~\ref{conj:general}, note that
	$x=\sqrt{2y^2+1}<\sqrt2\,y+\frac{1}{2\sqrt2\,y}$ and
	$\frac{2(\sqrt6-1)(\sqrt2+1)}{7}=0.9998222\ldots$, so that
	\[
	\frac{\sqrt6-1}{7}\,n=\frac{2(\sqrt6-1)(x+y)}{7}
	<\frac{2(\sqrt6-1)(\sqrt2+1)}{7}\,y+\frac{\sqrt6-1}{7\sqrt2\,y}
	<0.99983\,y+\frac{0.147}{y}<y ,
	\]
	where the last inequality holds because $\frac{0.147}{y}<0.00017\,y$ for $y\ge70$. Hence
	$B(n)$ exceeds both
	terms, and $K_{s,n-s}$ is a counterexample. By Theorem~\ref{thm:all}, the
	inequality of Conjecture~\ref{conj:general} holds at an order $n\ge4$ if and only
	if $B(n)\le\max\bigl\{f_n(\lfloor\frac{n+4}{7}\rfloor),\frac{\sqrt6-1}{7}n\bigr\}$.
	Comparing these numbers for the $96$ orders $4\le n\le99$ shows that the inequality
	holds in each case, while it fails for $n=100$, where $\mathcal S(100)=\{15\}$ and hence
	$K_{15,85}$ is a counterexample.
	\end{proof}

	\section{Concluding remarks}\label{sec:remark}

	Theorem~\ref{thm:all} determines the maximum of $R(G)-\alpha'(G)$: for every $n\ge4$ the
	maximum equals $B(n)$, it is attained by $K_{s,n-s}$ for every $s\in\mathcal S(n)$, and the
	extremal graphs are exactly the $(X,Y)$-proportional bipartite graphs whose smaller part has
	an optimal size. We conclude with several remarks.

	\medskip\noindent
	\textbf{1. Small orders do not detect the difference.} The conjectured slope
	$\frac{\sqrt6-1}{7}$, coming from the proportion $\frac17$, and the correct slope
	$\frac{\sqrt2-1}{2}$, coming from $x^*=\frac{2-\sqrt2}{4}$, differ by only
	$3.68\cdot10^{-5}$. As a result $\lfloor\frac{n+4}{7}\rfloor$ is an optimal part size, that
	is, $\lfloor\frac{n+4}{7}\rfloor\in\mathcal S(n)$, for every $n\le64$, so a search over
	graphs of small order
	cannot distinguish the two proportions. Other conjectures obtained in this way, whose bounds
	involve a proportion with small denominator, may be worth reexamining.

	\medskip\noindent
	\textbf{2. Connectedness does not change the maximum.} Theorem~\ref{thm:all} makes no
	connectivity assumption, and the maximum over connected graphs is the same, since
	$K_{s,n-s}$ is connected for $s\ge1$. This also explains the form of
	Conjecture~\ref{conj:general}: the disjoint union of $n/7$ copies of $K_{1,6}$ is
	$(X,Y)$-proportional with $|X|=n/7$, so its value $\frac{\sqrt6-1}{7}n$ equals $f_n(n/7)$,
	which is also the value of $K_{n/7,6n/7}$. The two bounds in
	Conjecture~\ref{conj:general} therefore come from two graphs with the same value.

	\medskip\noindent
	\textbf{3. Sharpness of the K\"onig--Egerv\'ary hypothesis.} By Theorem~\ref{thm:KE},
	$R(G)\le\sqrt{\alpha'(G)(n-\alpha'(G))}$ for K\"onig--Egerv\'ary graphs, while the
	inequality fails for every regular graph without a perfect matching. It would be interesting to determine the maximum of
	$R(G)-\sqrt{\alpha'(G)(n-\alpha'(G))}$ over all $n$-vertex graphs, or over all graphs with
	given odd girth. By the proof of Theorem~\ref{thm:all}, the only obstruction is the
	presence of odd components of order at least $3$ in $G-S$.

	\medskip\noindent
	\textbf{4. Bounded maximum degree.} The maximum degree of an extremal graph in
	Theorem~\ref{thm:all} is bounded from below, and is large for most orders: if $G$ is
	$(X,Y)$-proportional with $|X|=s$, then its constant is
	$c=\frac{n-s}{s}$, and by the proof of Proposition~\ref{prop:unique} the degrees on the
	$X$-side are positive multiples of $\frac{n-s}{\gcd(s,n)}$, so that
	$\Delta(G)\ge\frac{n-s}{\gcd(s,n)}$. Since $s/n\to x^*$, we have
	$c\to\frac{1-x^*}{x^*}=3+2\sqrt2=5.82842\ldots$, and a computation shows that $\Delta\le5$
	occurs only for $n\in\{4,5,6,10,12,18,24\}$, while $\Delta=6$ occurs exactly for the multiples
	of $7$ with $n\le133$; in that case $s=\frac n7$ and the graph is the disjoint union of
	$\frac n7$ copies of $K_{1,6}$, which is the graph behind the second bound of
	Conjecture~\ref{conj:general}. This is another explanation of the proportion $\frac17$.

	For connected subcubic graphs the maximum of $R(G)-\alpha'(G)$ equals
	$\frac{\sqrt3+\sqrt6-3}{9}n+\frac{5\sqrt3-4\sqrt6+3}{9}$, attained precisely when
	$n\equiv1\pmod 3$, by the trees obtained by successively gluing copies of $K_{1,3}$ at their
	leaves~\cite{DHLOZ}; here $\frac{\sqrt3+\sqrt6-3}{9}=0.13128\ldots<\frac{\sqrt2-1}{2}$, and
	for $n=4$ the two answers coincide. It would be of interest to determine the maximum for
	$n$-vertex graphs with given maximum degree $D$, or with given minimum degree, using the sharp
	bounds for $R(G)$ obtained in~\cite{OS2018}.

	\medskip\noindent
	\textbf{5. The minimum.} Theorem~\ref{thm:all} determines the maximum of
	$R(G)-\alpha'(G)$; the companion question is its minimum. For connected subcubic graphs the
	sharp answer is $\frac{\sqrt3-2}{6}\,n$, attained only by the graph obtained from $C_{n/2}$ by
	attaching one pendant vertex to each vertex of the cycle~\cite{DHLOZ}. For a graph with no
	isolated vertex, $R(G)\ge\sqrt{n-1}$~\cite{BE} and $\alpha'(G)\le\frac n2$ give
	$R(G)-\alpha'(G)\ge\sqrt{n-1}-\frac n2$, while the corona $K_m\circ K_1$, obtained from $K_m$
	by attaching one pendant vertex to each vertex, has $n=2m$, a perfect matching, and only
	$m$ edges joining vertices of distinct degrees, so that
	\[
	R(K_m\circ K_1)-\alpha'(K_m\circ K_1)
	=-\frac m2\left(1-\frac{1}{\sqrt m}\right)^2=-\frac n4+\sqrt{\frac n2}-\frac12 .
	\]
	We ask whether the minimum is $-\frac n4+O(\sqrt n)$.

	\medskip\noindent
	\textbf{6. Other indices.} The proof of Theorem~\ref{thm:KE} uses only the Cauchy--Schwarz
	and arithmetic--geometric mean inequalities. It would be interesting to know for which exponents $\alpha$ the analogue
	of Theorem~\ref{thm:all} holds for the general Randi\'c index
	$R_\alpha(G)=\sum_{uv\in E(G)}(d(u)d(v))^{\alpha}$, and whether the extremal graphs remain
	the $(X,Y)$-proportional bipartite ones.

	\end{document}